\documentclass{amsart}

\usepackage{amsmath,amssymb,amsthm,amscd}
\newtheorem{theorem}{Theorem}[section]
\newtheorem{lemma}{Lemma}[section]
\newtheorem{proposition}{Proposition}[section]
\newtheorem{corollary}{Corollary}[section]
\newtheorem{definition}{Definition}[section]
\newtheorem{remark}{Remark}[section]
\newtheorem{example}{Example}[section]

\begin{document}

\title{ON $1$-SEMIREGULAR AND $2$-SEMIREGULAR RINGS}

\author{Driss Bennis}
\address{Departement of Mathematics, Faculty of Sciences, Mohammed V University in Rabat, Rabat, Marocco}
\email{driss.bennis@um5.ac.ma;\quad driss${\_}$bennis@hotmail.com}

\author{Fran\c{c}ois Couchot}
\address{Normandie Univ, UNICAEN, CNRS, LMNO, 14000 Caen, France}
\email{francois.couchot@yahoo.com} 



\begin{abstract}
In this paper, we are mainly interested in the two questions ``which are the commutative  rings on which every finitely presented modules is $1$-periodic (resp., $2$-periodic)?". It is proved that these kinds of rings are particular cases of semiregular rings. So we call them $1$-semiregular and $2$-semiregular rings, respectively. We establish characterizations of these  rings in terms of various classical notions and we provides several examples of such rings.
\end{abstract}

\maketitle

\section{Introduction and preliminaries}
All rings considered in this paper are assumed to be commutative with an identity; in particular, $R$ denotes such a ring, and all modules are assumed to be unitary   modules. 
Recall that an $R$-module $M$ is said to be {\bf $n$-periodic}  if it admits a projective resolution of this form \[0\to M\to P_n\to \cdots \to P_1\to M\to 0\] (for a given positive integer $n$). Naturally, if $m$ is a multiple of $n$, an $n$-periodic module is $m$-periodic. It will be simply called {\bf periodic} if it is $1$-periodic. Notice that several authors use the term periodic to  design an $n$-periodic module for some positive integer $n$. According to  the terminology of  \cite{BIE}, $M$ is called {\bf $n$-F-periodic} (or simply {\bf F-periodic} for $n=1$)   if it admits a flat  resolution of this form \[0\to M\to F_n\to \cdots \to F_1\to M\to 0.\]

\medskip

 The existence and the structure of  $n$-periodic modules have been extensively investigated by many authors and in various areas of research such as Geometry and cohomology group theory and representation theory, where the associated projective resolution is usually supposed to be minimal  (see for instance the surveys \cite{ES, S06}  and the references within). In commutative ring theory, other interesting problems also occur such as relation between Betti numbers and $n$-periodic modules (see  for instance the overview \cite{PS} on other  open problems; see also \cite{BIE} for another new trend on this subject). In the classical representation theory of algebras, the following two questions have been intensively investigated (see \cite{ES, S06}):\\
\indent \textbf{Problem 1.} Given   an algebraically closed field $K$, determine the finite dimensional $K$-algebras $A$ for which every finitely generated module is $n$-periodic for some positive integer $n$.\\
\indent \textbf{Problem 2.} Assume $A$ is an algebra such that every finitely generated module is  $n$-periodic for some positive integer $n$. Does there exist a positive integer $m$ such that every finitely generated module is $m$-periodic? 

However, to our knowledge, there is no researcher in commutative algebra who has explicitly stated the following natural questions for a given integer $n\geq 1$:\medskip

\indent  \textbf{Question 1.} Which are the commutative  rings over which every module is $n$-periodic?\\
\indent   \textbf{Question 2.} Which are the commutative  rings over  which every finitely generated  module is $n$-periodic?\\
\indent  \textbf{Question 3.} Which are the commutative  rings over which every finitely presented   module is $n$-periodic?\\

We know  from the classical result \cite[Corollary 5.9]{FW} that a ring $R$ is quasi-Frobenius if  every finitely generated $R$-module embeds in a free $R$-module.  So we propose the following terminology:  

\begin{definition}\label{def-n-QF-ring}  Let $n$  be an integer $\geq 1$.   A ring $R$ is called {\bf $n$-quasi-Frobenius}  ({\bf weakly $n$-quasi-Frobenius})  if every $R$-module (resp., finitely generated  $R$-module)  is $n$-periodic. 
\end{definition}  

Then, $n$-quasi-Frobenius  rings and weakly $n$-quasi-Frobenius rings are both particular cases of quasi-Frobenius rings.  From \cite{BMO}, we can see that the $1$-quasi-Frobenius rings and the weakly $1$-quasi-Frobenius rings are the same and they are finite direct product of rings with at most one non-zero ideal. From \cite{BeHuWa15}, we deduce that the $2$-quasi-Frobenius rings and the weakly $2$-quasi-Frobenius rings are also the same and they are finite direct product of local Artinian valuation rings (equivalently, local principal ideal Artinian rings, \cite[Theorem 2.1]{EG}).  However,  we ignore how  $n$-quasi-Frobenius  rings and weakly $n$-quasi-Frobenius rings  are for the other cases.   In this paper, we turn our attention to question 3. Since these kinds of rings are particular cases of {\bf semiregular rings} (i.e., over which every module is embeded into a flat module  \cite{Mat85}), we introduce the following terminology.
 
\begin{definition}\label{def-12ring}  Let $n$  be an integer $\geq 1$.   A ring $R$ is called {\bf $n$-semiregular}   if every finitely presented $R$-module is  $n$-periodic. 
\end{definition}  
  
Indeed, the fact that an $n$-semiregular  ring is  semiregular can be deduced easily from   the following result.
  
\begin{proposition}\label{P:semireg}
Let $R$ be a ring. Then the following conditions are equivalent:
\begin{enumerate}
\item $R$ is semiregular;
\item each finitely presented $R$-module is contained in a free module;
\item each finitely presented cyclic $R$-module is contained in a free module;
\item $R$ is coherent and $(0:(0:I))=I$ for each finitely generated ideal $I$.
\end{enumerate} 
\end{proposition}
\begin{proof}
$(4)\Leftrightarrow (1)$ by \cite[Proposition 4.1]{Mat85}.

$(1)\Leftrightarrow (2)$ by \cite[Theorem 1]{Col75}.

$(2)\Rightarrow (3)$. Evident.

$(3)\Rightarrow (4)$. Let $I$ be a finitely generated ideal of $R$. The inclusion $I\subseteq (0:(0:I))$ always holds. There exist a free $R$-module $F$ of finite rank and a monomorphism $\varphi:R/I\rightarrow F$. Let $\{e_1,\dots,e_n\}$ be a basis of $F$. Then $\varphi(1+I)=a_1e_1+\dots+a_ne_n$. Let $A=Ra_1+\dots+Ra_n$. Clearly $I=(0:A)$. It sucessively follows that $A\subseteq (0:I)$, $(0:(0:I))\subseteq I$ and $(0:I)=A$. It remains to show that $I\cap J$ is finitely generated for each pair $I,\ J$ of finitely generated ideals. Let $A=(0:I)$ and $B=(0:J)$. Then $I\cap J=(0:A+B)$ which is finitely generated.
\end{proof}

\begin{definition}

An exact sequence of $R$-modules $0 \rightarrow F \rightarrow E \rightarrow G \rightarrow 0$  is {\bf pure}
if it remains exact when tensoring it with any $R$-module. Then, we say that $F$ is a \textbf{pure} submodule of $E$. When $E$ is flat, it is well known that $G$ is flat if and only if $F$ is a pure submodule of $E$. An $R$-module $E$ is {\bf FP-injective} if $\hbox{Ext}_R^1 (M, E) = 0$  for any finitely presented $R$-module $M$.  We recall that a module $E$ is FP-injective if and only if it is a pure submodule of every overmodule. 
\end{definition}
 
The $n$-semiregularity of rings is  preserved  under localizations and  also under quotient rings over pure ideals. 

\begin{proposition} \label{P:locS}  
Let $n$  be an integer $\geq 1$. If $R$ an $n$-semiregular ring., then we have:
\begin{enumerate}
\item  $S^{-1}R$ is $n$-semiregular for each multiplicative subset $S$ of $R$;
\item $R/I$ is $n$-semiregular for each pure ideal $I$ of $R$.
\end{enumerate}
\end{proposition}
\begin{proof}
$(1)$.  For each finitely presented $S^{-1}R$-module $M$ there exists a finitely presented $R$-module $M'$ such that $M\cong S^{-1}M'$. From $M'$ $n$-periodic over $R$, we easily get that $M$ is $n$-periodic over $S^{-1}R$.

$(2)$. It is well known that $R/I\cong S^{-1}R$ where $S=1+I$.
\end{proof}

We focus our intention in this paper on   $1$-semiregular rings  and $2$-semiregualr rings. In Section 2,  we give   a characterization of $1$-semiregular rings (Theorem \ref{T:G-sr}). We prove among other things that a ring $R$ is $1$-semiregular if and only if it is coherent and locally $1$-quasi-Frobenius  (that is, $PR_P$ is either zero or a simple module for each prime ideal $P$ of $R$). These  $1$-semiregular rings are particular cases of one-Krull-dimensional  B\'ezout rings (see condition 4 of Theorem \ref{T:G-sr}) and are exactly   coherent and arithmetical rings  such that the square ideals are  pure ideals. Then, as a consequence we get that  $1$-semiregular rings are elementary divisor rings (see Corollary \ref{C:Cyc-val}). In Section 3 we provide various examples of $1$-semiregular rings through characterizations of $1$-semiregularity of some ring constructions: direct product of rings, trivial extension of rings and amalgamated duplication of a ring along an ideal.  Section 4 is devoted to the study of $2$-semiregular rings.  After giving some preparatory results we give the first main result (Theorem \ref{T:2G-sr}) which characterizes   $2$-semiregular rings.  As corollaries we characterize   semilocal and local    $2$-semiregular rings (see Corollaries \ref{C:semilocal} and  \ref{C:Cyc-val}). We end the paper with   investigations of $2$-semiregularity of some ring constructions. This provides examples of $2$-semiregular rings and also helps in  discussing some conditions imposed before. 

Since modules over semiregular rings  embed into   flat modules, one can also ask  whether   $1$-semiregularity (resp., $2$-semiregularity) can be characterized with F-periodic (resp. $2$-F-periodic) modules.  Throughout this document, and to the extent that it allows, we provide results and examples showing that $1$-semiregular (resp., $2$-semiregular)  rings are closely related to rings on which every module is $1$-F-periodic (resp, $2$-F-periodic).

\section{$1$-semiregular rings}
Our aim in this section is to characterize  $1$-semiregular rings. We will show that there are related to some classical notions. For reader's convenience we recall these notions:

Recall that an $R$-module $E$ is said to be {\bf uniserial} if the set of its submodules is totally ordered by inclusion, and $E$ is called {\bf distributive} if $E_P$ is uniserial over $R_P$ for each maximal ideal $P$ of $R$. We say that $E$ is a {\bf coherent} $R$-module if each of its finitely generated submodules is finitely presented. A ring $R$ is a {\bf valuation ring} if $R$ is a uniserial module, and $R$ is {\bf arithmetical} if $R$ is locally a valuation ring.  We say that $R$ is an {\bf elementary divisor ring} if for every matrix $A$, with entries in $R$, there  exist a diagonal matrix $D$ and invertible matrices $P$ and $Q$, with entries in $R$, such that $PAQ = D$ .  A ring is a {\bf B\'ezout ring} if every finitely generated ideal is principal (see \cite{FI,FII} and also \cite{WiWi75} where the structure of finitely generated modules over B\'ezout  rings was investigated).\medskip

The following proposition is needed to prove the main result of this section.

\begin{proposition}\label{P:distr}
Let $R$ be a von Neumann regular ring. Then each coherent distributive module is isomorphic to an ideal of $R$.
\end{proposition}
\begin{proof}
Let $M$ be a coherent distributive $R$-module. For each $x\in M$, $Rx$ is finitely presented and projective. So, there exists an idempotent $e_x\in R$ such that $Rx\cong Re_x$. Let $I$ be the ideal generated by $\{e_x\mid x\in M\}$. For each $P\in D(I)$ we have $M_P\cong R_P\cong I_P$ and, since $V(I)\subseteq D(1-e_x)$ for each $x\in M$, for each $P\in V(I)$ $M_P=0=I_P$. We consider $E=\prod_{P\in D(I)}R_P$. There exist a monomorphism $\alpha:M\rightarrow E$ and a monomorphism $\beta:I\rightarrow E$. Let $M'=\alpha(M)$ and $I'=\beta(I)$. We easily check that $M'_P=(M'+I')_P=I'_P$ for each maximal ideal $P$ of $R$. We conclude that $M'=I'$ and consequently $M\cong I$.
\end{proof}

Now we can prove the main result of this section.

\begin{theorem} \label{T:G-sr} Let $R$ be a ring. The  following assertions are equivalent:
\begin{enumerate}
\item  $R$ is  $1$-semiregular;
\item each finitely presented cyclic $R$-module is periodic;
\item  $R$ is coherent and for each maximal ideal $P$, $PR_P$ is either zero or a simple module;
\item $R$ is a B\'ezout ring, each prime ideal is maximal and, for each $a\in R$ there exist two idempotents $e_1$ and $e_2$ such that \[R/aR\cong Re_1\oplus (Re_2/Rae_2)\ \mathrm{and}\ Re_2\cap (0:a)=Rae_2;\]
\item $R$ is coherent and arithmetical, and $I^2$ is a pure ideal for each ideal $I$ of $R$;
\item each cyclic $R$-module is F-periodic.
\end{enumerate} 
\end{theorem}
\begin{proof} 
It is obvious that $(1)\Rightarrow (2)$.

$(2)\Rightarrow (3)$. From $(2)$ we get that $R$ satisfies condition $(3)$ of Proposition \ref{P:semireg}. Hence $R$ is coherent. It is easy to check that $R_P$ verifies $(2)$ for each maximal ideal $P$. Now, assume that $R$ is local and let $P$ be its maximal ideal. Let $A$ be a finitely generated non-zero proper ideal. We have the following commutative diagram:
\[\begin{matrix}
 {} & {} & {} & {} & {} & {} & 0 & {} & {}\\
{} & {} & {} & {} & {} & {} & \downarrow & {} & {} \\
0 & \rightarrow & R/A & \rightarrow & F & \rightarrow & R/A & \rightarrow & 0 \\
 {} & {} & \downarrow & {} & \downarrow & {} & \downarrow & {} & {} \\
0 & \rightarrow & A & \rightarrow & R & \rightarrow & R/A & \rightarrow & 0 \\
{} & {} & {} & {} & {} & {} & \downarrow & {} & {} \\
{} & {} & {} & {} & {} & {} & 0 & {} & {} 
\end{matrix}\]
where $F$ is a free module and the right vertical map is the identity. Since $A\subseteq P$ the middle vertical homomorphism is surjective. It follows that the left vertical map is surjective too. Hence $A$ is principal and $R$ is a chain ring (valuation ring). Let $K$ be the kernel of the middle vertical homomorphism. It is a direct summand of $F$ and it is isomorphic to the kernel $L$ of the left vertical homomorphism. But $0$ is the sole projective submodule of $R/A$. Consequently $R/A\cong A$. We deduce that $A=(0:A)$. Let $a,b$ be two non-zero elements of $P$. We may assume that $Ra\subseteq Rb$. It follows that $Rb=(0:b)\subseteq (0:a)=Ra$. Hence $Ra=Rb=P$.

$(3)\Rightarrow (4)$. Clearly $R$ is zero-Krull-dimensional and arithmetical. By \cite[Theorem III.6]{Cou03} $R$ is B\'ezout. So, $(0:a)=Rb$ for some $b\in R$. It follows that $D(b)=D(e_1)$, $V(a)\cap V(b)=D(e_2)$ and $D(a)=D(1-e_1-e_2)$ where $e_1$ and $e_2$ are idempotent. We easily deduce that $ae_1=0$, $Ra(1-e_1-e_2)=R(1-e_1-e_2)$ and $Rae_2=Rbe_2$. Consequently the conditions  of the assertion $(4)$ hold.

$(4)\Rightarrow (1)$. 
 Let $M$ be a finitely presented $R$-module. By \cite[Theorem III.6]{Cou03}, $R$ is an elementary divisor ring. So, $M$ is a direct sum of finitely presented cyclic modules by \cite[(3.8)Theorem]{LLS74}. It is easy to see that any direct sum of periodic modules is periodic too. So we may assume that $M=R/Ra$ for some $a\in R$. Now since $Re_1$ is projective it is periodic. Clearly $Rae_2\cong Re_2/Rae_2$. Hence $Re_2/Rae_2$ is periodic too. 

$(2)\Rightarrow (5)$. Let $I$ be an ideal of $R$. Since, for each maximal ideal $P$, either $PR_P=0$ or it is a simple module, the ideal $(I^2)_P$ is either $R_P$ or $0$. Hence $I^2$ is a pure ideal. 

$(5)\Rightarrow (2)$. Let $P$ be a maximal ideal of $R$. Since $R/P^2$ is local and a flat module over $R$ we have $R_P=R/P^2$, and it is a valuation ring. We easily conclude that either $PR_P=0$ or $PR_P$ is a simple module.

Let us observe that the equivalence of the first five conditions is shown.

$(6)\Rightarrow(2)$ is obvious. 

$(1)\Rightarrow (6)$. Let $A$ be an ideal of $R$. If $A=A^2$ then $A$ is a pure ideal by $(5)$. In this case $R/A$ is flat. So, in the sequel we assume that $A\ne A^2$. Let $R'=R/A^2$, $A'=A/A^2$ and $J$ be the Jacobson radical of $R'$. We have $J^2=0$ and $A'\subseteq J$. So, $A'$ is a module over the von Neumann regular ring $R'/J$. Clearly, $A'$ is distributive, and we easily check that it is also coherent over $R'/J$. By Proposition \ref{P:distr} $A'$ is isomorphic to an ideal $B$ of $R'/J$. Let $I$ be the pure ideal of $R'$ which satisfies $I_P=0$ if and only if $A'_P=0$. We easily check that $A'\subseteq I$ and $B=I/I\cap J$. Clearly $I\cap J=A'$. Since $A'\cong I/A'$ and $R/A\cong R'/A'$ there is a monomorphism $\phi:A'\rightarrow R/A$ and its image is of the form $I'/A$ ($I=I'/A^2$), where $I'$ is a pure ideal of $R$. We consider the following commutative pushout diagram:
\[\begin{matrix}
{} & {} & 0 & {} & 0 & {} & 0 & {} & {} \\
{} & {} & \downarrow & {} & \downarrow & {} & \downarrow & {} & {} \\
0 & \rightarrow & A' & \rightarrow & R' & \rightarrow & R/A & \rightarrow & 0 \\
{} & {} & {\scriptstyle{\phi}}\downarrow & {} & \downarrow & {} & \downarrow & {} & {} \\
0 & \rightarrow & R/A & \rightarrow & F & \rightarrow & R/A & \rightarrow & 0 \\
{} & {} & \downarrow & {} & \downarrow & {} & \downarrow & {} & {} \\
0 & \rightarrow & R/I' & \rightarrow & R/I' & \rightarrow & 0 & {} & {} \\
{} & {} & \downarrow & {} & \downarrow & {} & {} & {} & {} \\
{} & {} & 0 & {} & 0 & {} & {} & {} & {} 
\end{matrix}\]
By condition $(5)$ $R'$ is flat. Since $I'$ is a pure ideal, $R/I'$ is flat too. It follows that $F$ is flat and $R/A$ is F-periodic.
\end{proof}

From \cite[Theorem III.6]{Cou03} and Proposition \ref{P:semireg} we deduce the following:
\begin{corollary}\label{C:edr}
Each $1$-semiregular ring $R$ is an elementary divisor ring and a semiregular ring.
\end{corollary}

\begin{proposition}\label{P:fact}
Let $R$ be a ring and $A$ a proper finitely generated ideal of $R$. If $R$ is $1$-semiregular then $R/A$ is $1$-semiregular too.
\end{proposition}
\begin{proof}
If $R$ is $1$-semiregular then $R/A$ is coherent. Now, it is easy to check that the maximal ideal of $(R/A)_{P/A}$ is either zero or a simple module for each $P\in V(A)$. Whence $R/A$ is $1$-semiregular.
\end{proof}

\section{Some constructions of $1$-semiregular rings}
In this section, we show when some classical ring constructions are  $1$-semiregular. This provides  a family of examples of $1$-semiregular rings and allows to shed light on conditions  assumed in some results of Section 2.\medskip

We start with the direct product of rings.

\begin{proposition}\label{P:produit}
Let $(R_{\lambda})_{\lambda\in\Lambda}$ be a family of rings and $R=\prod_{\lambda\in\Lambda}R_{\lambda}$. Then $R$ is $1$-semiregular if and only if $R_{\lambda}$ is $1$-semiregular for each $\lambda\in\Lambda$.
\end{proposition}
\begin{proof} For $\lambda\in\Lambda$ let $e_{\lambda}=(\delta_{\lambda,\mu})_{\mu\in\Lambda}$, where $\delta_{\lambda,\mu}$ is the Kronecker symbol. Since $R_{\lambda}\cong R/(1-e_{\lambda})R$, $R_{\lambda}=S^{-1}_{\lambda}R$ where $S_{\lambda}=\{e_{\lambda}\}$. So, if $R$ is $1$-semiregular then $R_{\lambda}$ is $1$-semiregular by Proposition \ref{P:locS}.

Conversely, let $C$ be a finitely presented cyclic $R$-module and for each $\lambda\in\Lambda$ let $C_{\lambda}=R_{\lambda}\otimes_RC$. By \cite[12.9]{Wis91} $C\cong\prod_{\lambda\in\Lambda}C_{\lambda}$. By Theorem \ref{T:G-sr}(4), for each $\lambda\in\Lambda$, there exists a direct summand $F_{\lambda}$ of $R^3_{\lambda}$ and an exact sequence  \[0\rightarrow C_{\lambda}\rightarrow F_{\lambda} \rightarrow C_{\lambda}\rightarrow 0.\] Let $F=\prod_{\lambda\in\Lambda}F_{\lambda}$. Then $F$ is a direct summand of $R^3$ and $C$ is periodic. We conclude by Theorem \ref{T:G-sr}(2).
\end{proof}

The following example shows that $R/A$ is not generally $1$-semiregular if we do not assume that $A$ is finitely generated in Proposition \ref{P:fact}.

\begin{example}
Let $V$ be a local $1$-quasi-Frobenius ring which is not a field, $vV$ its maximal ideal and $\Lambda$ an infinite set. We put $R=V^{\Lambda}$. For each $\lambda\in\Lambda$ we define $e_{\lambda}$ as in the proof of Proposition \ref{P:produit}. Let $A=v(\sum_{\lambda\in\Lambda}Re_{\lambda})$. Then $R$ is $1$-semiregular but not $R/A$. (If we put $\bar{r}=r+A$ for each $r\in R$ and $\overline{R}=R/A$ then we have $(0: \bar{v})=\overline{R}\bar{v}+(\sum_{\lambda\in\Lambda}\overline{R}\bar{e}_{\lambda})$.)
\end{example}

Let $A$ be a ring and $E$  an $A$-module. The {\bf trivial ring extension} of $A$ by $E$ (also called the idealization of $E$ over $A$) is the ring $R := A\propto E$ whose underlying group is $A\times E$ with multiplication given by $(a, e)(a', e') = (aa', ae' + a'e)$ (for more details about trivial ring extension, see, for instance, \cite{AW}). In Theorem \ref{T:triv}, we characterize when a trivial ring extension is $1$-semiregular.  For this, we need the following  two lemmas and  proposition.

\begin{lemma}\label{L:princ}
Let $A$ be a von Neumann regular ring and $E$ an $A$-module such that $E_P$ is a vector space of rank $\leq 1$ over $A_P$ for each maximal ideal $P$. Let $0\ne x\in E$ and let $\epsilon$ be the idempotent of $A$ satisfying $(0:x)=A(1-\epsilon)$. Then $\epsilon E=Ax$.
\end{lemma}
\begin{proof}
It  is clear that $\epsilon E_P=A_Px=0$ if $P\in D(1-\epsilon)$. Let $P\in D(\epsilon)$. Since $(0:x)\subseteq P$ we have $A_Px\ne 0$. But $E_P$ is simple so $\epsilon E_P=E_P=A_Px$. We conclude that $\epsilon E=Ax$.
\end{proof}

\begin{proposition}\label{P:zeroKrull}\textnormal{\cite[Proposition 3.8]{Couc16}}
Let $A$ be a ring whose prime ideals are  maximal. Let $X$ be the set of all maximal ideals $P$ such that $PA_P=0$. We denote by $I$,  the kernel of the naturel map $A\rightarrow \prod_{P\in X}A_P$. Then $A/I$ is von Neumann regular and if, in addition $A$ is coherent, then $I$ is a pure submodule of $A$ and $X=V(I)$.
\end{proposition}

\begin{lemma}\label{L:ann}
Let $A$ be a $1$-semiregular ring, $E$ a coherent $A$-module such that $\mathrm{Supp}(E)\subseteq V(I)$ where $I$ is the ideal of $A$ defined in Proposition \ref{P:zeroKrull}. Assume that $E_P$ is a simple module for each maximal ideal $P$. Then, for any $a\in A$ and $x\in E$ there exists $y\in (0:_Ea)$ such that $(0:_Ea)=Acy+bE$ where $Ab=(0:a)\cap (0:x)$ and $Ac=(0:b)$.
\end{lemma}
\begin{proof}
Let $A'=A/I$. For each $r\in A$ we denote by $\bar{r}$ the image of $r$ in $A'$. Since $A'$ is von Neumann regular then $A'\bar{a}=A'e$ where $e$ is an idempotent of $A'$. Since $E$ is a module over $A'$, $E$ is FP-injective, whence $(0:_Ea)=(1-e)E$. The submodule $Ax$ is finitely presented and flat over $A'$ and $A$. It follows that $Ax$ is isomorphic to a direct summand of $A$. So, there exists an idempotent $\epsilon\in A$ such that $(0:x)=A(1-\epsilon)$.
It follows that $(1-e)E=\epsilon(1-e)E+(1-\epsilon)(1-e)E$. By Lemma \ref{L:princ} $\epsilon E=Ax$. We get that $\epsilon(1-e)E=A\epsilon(1-e)x$. From $\bar{b}=(1-e)(1-\epsilon)$ we get $\bar{c}=(e+\epsilon-e\epsilon)$. We take $y=(1-e)\epsilon x$. We easily check that $y=cy$. Therefore, $(0:_Ea)=Acy+bE$.
\end{proof}

\begin{theorem}\label{T:triv}
Let $A$ be a ring, $E$ a non-zero $A$-module and $R=A\propto E$ the trivial ring extension of $A$ by $E$. The following two conditions are equivalent:
\begin{enumerate}
\item $R$ is $1$-semiregular;
\item $A$ and $E$ verify the following properties:
\begin{enumerate} 
\item $A$ is $1$-semiregular,
\item $A_P$ is a field for each $P\in\mathrm{Supp}(E)$,
\item $E_P$ is a simple module for each $P\in\mathrm{Supp}(E)$,
\item $E$ is FP-injective and coherent.
\end{enumerate}   
\end{enumerate}
\end{theorem}
\begin{proof}
$(1)\Rightarrow (2)$. We use the notations of Proposition \ref{P:zeroKrull}. From Corollary \ref{C:edr} and \cite[Corollary 1.4]{Cou15} we deduce that $A$ is arithmetical, $E$ is FP-injective, $E_P$ is uniserial for each maximal ideal $P$ of $A$ and $\mathrm{Supp}(E)\subseteq X$. We easily see that $PA_P$ is simple for each maximal ideal $P$ of $A$. Let $a\in A$. Then $(0:(a,0))=R(b,x)$ for some $b\in A$ and $x\in E$ because $R$ is coherent. We easily deduce that $(0:a)=Ab$. By \cite[Corollary 1.11]{ShWi74} $A$ is coherent and by Theorem \ref{T:G-sr} $A$ is $1$-semiregular. 

Let $0\ne x\in E$. Then $(0:(0,x))=R(a,y)$ for some $a\in A$ and $y\in E$. It follows that $(0:x)=Aa$. Let $M$ be a finitely generated submodule of $E$. Then $M_P$ is cyclic for each maximal ideal $P$. From \cite[Theorem 2.1]{WiWi75} we deduce that $M$ is cyclic. Hence $M$ is finitely presented and $E$ is coherent.

$(2)\Rightarrow (1)$. It is easy to check that $PR_P$ is simple for each maximal ideal $P$ of $R$. It remains to show that $R$ is coherent. By \cite[Corollary 1.4]{Cou15} $R$ is arithmetical. Let $(a,x)\in R$. If $(a,x)(d,z)=0$ with $d\in A$ and $z\in E$ then $ad=0$ and $dx+az=0$. We use the same notations as in Lemma \ref{L:ann}. Then $A'\bar{a}=A'e$ where $e$ is an idempotent of $A'$. We deduce that $\bar{d}\in (0:\bar{a})= A'(1-e)$.  On the other hand $E$ is a module over $A'$. We have $aE\cap dE=0$. Hence $dx=0$ and $az=0$. We get that $d\in (0:a)\cap (0:x)=Ab$ and $z\in (0:_Ea)$. By using Lemma \ref{L:ann} there exists $y\in (0:_Ea)$ such that $(0:_Ea)=Acy+bE$ where $Ac=(0:b)$. We have $d=rb$ for some $r\in A$ and $(d,ry)=(b,y)(r,0)$. On the other hand $z-ry=scy+bw$ where $s\in A$ and $w\in E$. We deduce that $(0,z-ry)=(b,y)(sc,w)$. We conclude that $(0:(a,x))=R(b,y)$. Hence $R$ is coherent by \cite[Corollary 1.11]{ShWi74}. 
\end{proof}

\begin{corollary}\label{C:triv}
Let $A$ be a von Neumann regular ring. Then $A\propto B$ is $1$-semiregular for each ideal $B$.
\end{corollary}

Let $R$ be a ring such that each $R$-module is F-periodic. It is obvious that $R$ is $1$-semiregular. Clearly, if $R$ is either $1$-quasi-Frobenius or von Neumann regular then each $R$-module is F-periodic. We will give other examples of rings $R$ for which each $R$-module is F-periodic. 

Let $A$ be a ring and $R=A\propto A$. We do as in \cite{PaRo73}. Let $U$ be an $R$-module and $f:U\rightarrow U$ the $A$-homomorphism defined by $f(u)=(0,1)u$ for each $u\in U$. Clearly $f^2=0$. Conversely, if we consider the pair $(U,f)$ where $U$ is an $A$-module and $f$ an $A$-endomorphism of $U$ such that $f^2=0$ then we get that $U$ is an $R$-module by setting $(a,b)u=au+bf(u)$ for any $a,b\in A$ and $u\in U$. Let $(U,f)$ and $(V,g)$ be two $R$-modules. An $A$-homomorphism $\alpha:U\rightarrow V$ is an $R$-homomorphism if and only if $g\alpha=\alpha f$.

The following proposition is an immediate consequence of \cite[Proposition 4.1]{PaRo73}. But in our particular case there is a simple proof.

\begin{proposition} \label{P:Rflat}
Let $A$ be a von Neumann regular ring and $R=A\propto A$. An $R$-module $(U,f)$ is flat if and only if $\ker(f)=\mathrm{Im}(f)$.
\end{proposition}
\begin{proof}
Since $R$ is B\'ezout, $U$ is flat over $R$ if and only if, for each $r\in R$ the map $Rr\otimes_RU\rightarrow U$, induced by the inclusion $Rr\subseteq R$, is a monomorphism. Every element of $Rr\otimes U$ is of the form $r\otimes u$, where $u\in U$. By \cite[Proposition I.\S 2.13]{Bou61}, $U$ is flat if and only if, for every $u\in U$ and $r\in R$ such that $ru=0$, $u\in (0:r)U$.

Suppose that $(U,f)$ is flat. Let $u\in U$ such that $f(u)=0$. This means that $(0,1)u=0$. Since $(U,f)$ is flat over $R$ and the annihilator of $(0,1)$ is generated by $(0,1)$ we have $u=(0,1)v=f(v)$ for some $v\in U$.

Now assume that $\ker(f)=\mathrm{Im}(f)$. Suppose that $(a,a')u=0$ where $a,a'\in A$ and $u\in U$. Let $e,e'$ the idempotents of $A$ such that $(0:a)=Ae$ and $(0:a')=Ae'$. We have $au+a'f(u)=0$. We deduce that $af(u)=0$. Since $U$ is flat over $A$, $f(u)=ev$ for some $v\in U$. It follows that $f(u)=ef(u)$. On the other hand $a=(1-e)a$. So, we have $(1-e)au+a'ef(u)=0$. Consequently $au=0$ and $a'f(u)=0$. We get that $u=eu$ and $f(u)=e'f(u)=ee'f(u)$. Hence $f(u-ee'u)=0$. We deduce that $u-ee'u=f(w)$ for some $w\in U$. It follows that \[u=eu=ee'u+ef(w)=(ee',0)u+(0,e)w\] with $(a,a')(ee',0)=0$ and $(a,a')(0,e)=0$. Hence $(U,f)$ is flat over $R$.
\end{proof}

\begin{proposition} \label{P:Fperiodic}
Let $A$ and $R$ as in Proposition \ref{P:Rflat}. We assume that the characteristic of $A$ is $2$. Then each $R$-module is F-periodic.
\end{proposition}
\begin{proof}
Let $(U,f)$ be an $R$-module. We consider $G=U\times U$ and $g:G\rightarrow G$ the homomorphism defined by $g((u,v))=(v,0)$ for each $(u,v)\in G$. It is easy to see that $\ker(g)=\mathrm{Im}(g)$. So, $G$ is flat over $R$ by Proposition \ref{P:Rflat}. Now, we consider $\alpha:G\rightarrow U$ defined by $\alpha((u,v))=f(u)+v$ for each $(u,v)\in G$. It is obvious that $\alpha$ is surjective and we easily check that $f\alpha=\alpha g$. Hence $\alpha$ is an epimorphism over $R$. We have \[\ker(\alpha)=\{(u,-f(u))\mid u\in U\}=\{(u,f(u))\mid u\in U\},\] since $2$ is the characteristic of $A$. Let $\beta:U\rightarrow G$ be the homomorphism defined by $\beta(u)=(u,f(u))$ for each $u\in U$. Clearly $\beta$ is injective and $\beta f=g\beta$. So, $\beta$ is an $R$-monomorphism and the following sequence of $R$-modules  \[0\rightarrow (U,f)\xrightarrow{\beta} (G,g)\xrightarrow{\alpha} (U,f)\rightarrow 0\] is exact. Hence $(U,f)$ is F-periodic.
\end{proof}

\begin{remark}
By \cite[Proposition 4.1]{PaRo73} Proposition \ref{P:Fperiodic} holds even if $A$ is not commutative.
\end{remark}

Let $R$ be a ring and $I$ an ideal of $R$. The {\bf amalgamated duplication of $R$ along $I$}, denoted by $R\bowtie I$, is the subring of $R\times R$ given by
\[R\bowtie I:=\{(r, r + a) \mid r\in R, a\in I\}.\]

\begin{theorem}\label{T:FP}
Let $R$ be a ring and $I$ an ideal. Then $R\bowtie I$ is self FP-injective (resp. semiregular) if and only if $I$ is a pure ideal and $R$ is self FP-injective (resp. semiregular).
\end{theorem}
\begin{proof}
We use the fact that a commutative ring is semiregular if and only if it is coherent and self FP-injective (see \cite[Proposition 3.3]{Mat85}).

Assume that $R\bowtie I$ is self FP-injective. 

Consider the equation $rx=a$ where $r\in R$ and $a\in I$. Suppose that there is a solution in $R$. Then the equation $(r,r)X=(0,a)$ has a solution in $R\times R$. So, there is a solution $(s,s+b)$ in $R\bowtie I$ and it follows that $rs=0$ and $rb=a$.

Since $R\bowtie I$ is flat over $R$, it is FP-injective over $R$. Then $R$ is self FP-injective ($R$ is isomorphic to a direct $R$-summand of $R\bowtie I$). It follows that $R\bowtie I$ is faithfully flat over $R$. So, if $R\bowtie I$ is coherent then $R$ is also coherent.

Conversely we show that $0\times I$ is a pure ideal of $R\bowtie I$. Consider the equation $(r,r+a)X=(0,b)$ where $r\in R$ and $a,\ b\in I$ and suppose there is a solution $(x,x+y)$ in $R\bowtie I$. Then $rx=0$ and the equation $(r+a)Y=b$ has a solution  in $R$. Since there is a solution $z\in I$ it follows that $(0,z)$ is a solution of $(r,r+a)X=(0,b)$. Hence $R\times 0\cong R\bowtie I/0\times I$ is flat over $R\bowtie I$. Every element $(r,r+a)$ of $R\bowtie I$ can be written $((r+a)-a,(r+a))$. In the same way we show that $I\times 0$ is a pure ideal of $R\bowtie I$ and $0\times R\cong R\bowtie I/I\times 0$ is flat over $R\bowtie I$. We successively conclude that $R\times R$ is flat and FP-injective over $R\bowtie I$. Now consider the following system of equations:
\begin{equation}\label{eq}
\sum_{1\leq i\leq n}(r_{i,j},r_{i,j}+a_{i,j})X_i=(s_j,s_j+t_j),\ 1\leq j\leq p,\ r_{i,j},s_j\in R,\ a_{i,j},t_j\in I.
\end{equation}
Suppose there is a solution $(x_i,y_i)_{1\leq i\leq n}$ in $R\times R$. Then we have:
\begin{equation}
\sum_{1\leq i\leq n}r_{i,j}x_i=s_j\ \mathrm{and}\ \sum_{1\leq i\leq n}(r_{i,j}+a_{i,j})y_i=s_j+t_j,\ 1\leq j\leq p. 
\end{equation}
We deduce that 
\begin{equation}
\sum_{1\leq i\leq n}(r_{i,j}+a_{i,j})(y_i-x_i)=t_j-\sum_{1\leq i\leq n}a_{i,j}x_i,\ 1\leq j\leq p. 
\end{equation}
This last system of equations has a solution $(z_i)_{1\leq i\leq n}$ in $I$. We easily check that $(x_i,x_i+z_i)_{1\leq i\leq n}$ is solution of (\ref{eq}). It follows that $R\bowtie I$ is self FP-injective. We get that $R\times R$ is faithfully flat over $R\bowtie I$. Hence $R\bowtie I$ is coherent if $R$ is coherent.
\end{proof}

\begin{theorem}
Let $A$ be a ring, $I$ an ideal and $R=A\bowtie I$. Then $R$ is $1$-semiregular if and only if $I$ is a pure ideal and $A$ is $1$-semiregular.
\end{theorem}
\begin{proof} If $R$ is $1$-semiregular then it is semiregular by Corollary \ref{C:edr}. So, $I$ is a pure ideal of $A$ and $0\times I$ a pure ideal of $R$ by Theorem \ref{T:FP} and its proof. Consequently $A$ is $1$-semiregular by Corollary \ref{C:edr}.

Conversely, $R$ is semiregular by Theorem \ref{T:FP}. Let $P$ be a maximal ideal of $R$ and $L$ the inverse image of $P$ by the homomorphism $\phi:A\rightarrow R$ defined by $\phi(r)=(r,r)$. Then $R_P=(R_L)_{P_L}$. If $I_L=0$ then $R_L=R_P=A_L$. Else $I_L=A_L$ and $R_L=A_L\times A_L$. It follows that $R_P=A_L$. Hence $PR_P$ is a simple module for each maximal ideal $P$ of $R$.
\end{proof}

\section{$2$-semiregular rings}
Now we turn our attention to $2$-semiregular rings.  We will give a characterization of such rings. To this end we give some preparatory results.

\begin{proposition}
\label{P:2sr} Let $R$ be a semiregular ring. Then $R$ is $2$-semiregular if $R$ is an elementary divisor ring.
\end{proposition}
\begin{proof}
Since $R$ is B\'ezout and semiregular, then for each $a\in R$ there exists $b\in R$ such that $Rb=(0:a)$ and $Rb\cong R/Ra$. So, if $\tilde{a}$ is the multiplication by $a$ in $R$ we have the following exact sequence:
\[0\rightarrow R/aR\rightarrow R\xrightarrow{\tilde{a}} R\rightarrow R/aR\rightarrow 0.\]
Now we conclude by using the fact that each finitely presented $R$-module is a direct sum of cyclic modules because $R$ is an elementary divisor ring.
\end{proof}

The following lemma generalizes \cite[Lemma 2.3]{BeHuWa15}.

\begin{lemma}\label{L:2F} Let $R$ be a coherent ring and $W$a $2$-F-periodic module. Then there exists an exact sequence $0\rightarrow W\rightarrow G\rightarrow F\rightarrow W\rightarrow 0,$ such that:
\begin{enumerate}
\item $F$ is free and $G$ is flat;
\item $F$ is finitely generated if $W$is finitely generated, and they have the same minimal number of generators;
\item $G$ is projective and finitely generated if $W$is finitely presented.
\end{enumerate} 
\end{lemma}
\begin{proof} There exists an sequence $0\rightarrow W\rightarrow F_1\xrightarrow{\phi} F_0\xrightarrow{\pi} W\rightarrow 0$, where $F_1$ and $F_0$ are flat $R$-modules. Let $\alpha:F\rightarrow W$ be an epimorphism where $F$ is a free $R$-module (of rank equal to the minimal number of generators of $W$ if it is finitely generated), $L=\ker(\alpha)$ and $K=\ker(\pi)$. There exists a homomorphism $\beta:F\rightarrow F_0$ such that $\alpha=\pi\beta$. Let $\gamma:F_1\oplus F\rightarrow F_0$ be the linear map defined by $\gamma((x,y))=\phi(x)+\beta(y)$ where $x\in F_1$ and $y\in F$ and let $G=\ker(\gamma)$. We easily check that $\gamma$ is surjective. As $F_0$ is flat, $G$ is a pure submodule of $F_1\oplus F$. So, $G$ is flat. We get the following commutative diagram with exact rows and exact columns:

\begin{equation}\label{E:diag}
\begin{matrix}
 {} & {} & 0 & {} & 0 & {} & 0 & {} & {}\\
{} & {} & \downarrow & {} & \downarrow & {} & \downarrow & {} & {} \\
0 & \rightarrow & W & \rightarrow & G & \rightarrow & L & \rightarrow & 0 \\
 {} & {} & \downarrow & {} & \downarrow & {} & \downarrow & {} & {} \\
0 & \rightarrow & F_1 & \rightarrow & F_1\oplus F & \rightarrow & F & \rightarrow & 0 \\
{} & {} & {\scriptstyle{\phi}}\downarrow & {} &{\scriptstyle{\gamma}} \downarrow & {} & {\scriptstyle{\alpha}} \downarrow & {} & {} \\
0 & \rightarrow & K & \rightarrow & F_0 & \xrightarrow{\pi} & W & \rightarrow & 0 \\
{} & {} & \downarrow & {} & \downarrow & {} & \downarrow & {} & {} \\
{} & {} & 0 & {} & 0 & {} & 0 & {} & {} 
\end{matrix}
\end{equation}

Now, we combine the top horizontal sequence with the right vertical one to get the desired one. Whenever $W$ is finitely presented, we choose $F$ of finite rank. Since $R$ is coherent we have that $L$ is finitely presented. It follows that $G$ is finitely presented, whence it is projective. 
\end{proof}

As a consequence we get:

\begin{corollary}\label{C:2F} Let $R$ be a ring. Then $R$ is $2$-semiregular if each $R$-module is 
$2$-F-periodic. \end{corollary}

We will show at the end of the paper that there is  a $2$-semiregular ring with a non 
$2$-F-periodic module.\\

We define the {\bf Fitting invariants} $F_i(A)$ of an $m\times n$ matrix $A$ as follows: $F_i(A)$ is the ideal of $R$ generated by the $(m-i)\times (m-i)$ minors of $A$, for $0\leq i< m$ and $F_i(A)=R$ for $i\geq m$. If $A$ names the module $M$, we put $F_i(M)=F_i(A)$. The following well-known proposition will be used later:

\begin{proposition}\label{P:Fitt}\textnormal{\cite[Proposition 20.8]{Eis95}}
Let $M$ be a finitely presented $R$-module. Then $M$ is a projective module of constant rank $k$ if and only if $F_{k-1}(M)=0$ and $F_k(M)=R$.
\end{proposition}

\begin{lemma}\label{L:proj}
Let $R$ be an arithmetical ring and $r\in R$. Assume that $(0:r)=Ra+Rb$. Then there exists an exact sequence $0\rightarrow R/Rr\xrightarrow{\alpha} G\xrightarrow{\pi} Rr\rightarrow 0$ where $G$ is a projective module of rank $1$ if and only if there exist $u,v\in R$ and $s,t\in (0:r)$ such that $bu=av,\ su=a(1+t)$ and $sv=b(1+t)$.
\end{lemma}
\begin{proof}
{\bf only if}. Let $g=\alpha(1+Rr)$ and $h\in G$ such that $\pi(h)=r$. Clearly $g,h$ generate $G$, $(0:g)=Rr$ and $(Rg:h)=(0:r)$. So, there exist $u,v\in R$ such that $ah+ug=0$ and $bh+vg=0$, and $G$ is named by the matrix $A=\binom{r\ u\ v}{0\ a\ b}$. Since $G$ is projective of constant rank $1$ we have $F_0(G)=0$ and $F_1(G)=R$. We deduce that $av=bu$ and $Rr+Ra+Rb+Ru+Rv=R$. Whence there exist $r',a',b',u',v'\in R$ such that $r'r+a'a+b'b+u'u+v'v=1$. We get the following matrix equality: $\binom{b\ -a}{u'\ v'}\binom{u}{v}=\binom{0}{1-r'r-a'a-b'b}$. By multiplying on the left by $\binom{v'\ a}{-u'\ b}$ we obtain: $(au'+bv')u=a(1-r'r-a'a-b'b)$ and $(au'+bv')v=b(1-r'r-a'a-b'b)$. If we put $s=(au'+bv')$ and $t=-(a'a+b'b)$ we get the required relations.

{\bf if}. Let $G$ be the finitely presented module named by the matrix $A=\binom{r\ u\ v}{0\ a\ b}$. Consequently $G$ is generated by $g, h$ with the relations $rg=0,\ ug+ah=0$ and $vg+bh=0$. We get $Rg\cong R/Rr$ and $G/Rg\cong Rr$. It remains to show that $G$ is projective. The equalities $ra=0,\ rb=0$ and $bu=av$ implies that $F_0(G)=0$. Let $P$ be a maximal ideal. If $P\in D(r)$  then $F_1(G)_P\supseteq R_Pr=R_P$. If $P\in D((0:r))$  then $F_1(G)_P\supseteq (R_Pa+R_Pb)=R_P$. Now assume that $P\in V(r)\cap V((0:r))$. In this case $(1+t)$ is a unit of $R_P$. Consequently $(0:r)_P=R_Ps$. But since $R_P$ is a valuation ring, either $(0:r)_P=R_Pa$ or $(0:r)_P=R_Pb$. It follows that either $u$ or $v$ is a unit of $R_P$. So, $F_1(G)_P\supseteq (R_Pu+R_Pv)=R_P$. We conclude that $F_1(G)=R$ and $G$ is projective of rank $1$ by Proposition \ref{P:Fitt}.
\end{proof}

\begin{lemma}\label{L:summand}Let $U,\ U',\ V,\ V'$ and $G$ be modules over a ring $R$. Assume there exists an exact sequence $0\rightarrow U\oplus V\rightarrow G\xrightarrow{\alpha} U'\oplus V'\rightarrow 0$. Let $G'$ be the inverse image of $U'$ by $\alpha$. If $G=H\oplus J$, with $U\subseteq H$, $V\subseteq J$, $\alpha(H)=U'$ and $\alpha(J)=V'$, then $G'/V\cong H$. 
\end{lemma}
\begin{proof}
Clearly $H\subseteq G'$ and $H\cap V\subseteq H\cap J=0$. So there is a monomorphism $\lambda$ from $H$ into $G'/V$. Now, let $g\in G'$. Since $\alpha(g)\in U'$ and $\alpha(H)=U'$ there exists $h\in H$ such that $g-h\in\ker(\alpha)=U\oplus V$. So, there exists $u\in U$ such that $g\equiv (h+u)$ modulo $V$. Hence $\lambda$ is an isomorphism.
\end{proof}

\begin{theorem} \label{T:2G-sr} Let $R$ be a ring. The following two conditions are equivalent:
\begin{enumerate} 
\item  $R$ is  $2$-semiregular; 
\item 
Every finitely presented cyclic $R$-module is $2$-periodic;
\item $R$  verifies the following properties: 
\begin{enumerate}
\item $R$ is arithmetical and semiregular;
\item for each $r\in R$, $(0:r)$ is generated by at most two elements. Moreover, if $(0:r)=Ra+Rb$ there exist $u, v\in R$ and $s, t\in (0:r)$ such that $av=bu,\ su=a(1+t)$ and $sv=b(1+t)$.
\end{enumerate}
\end{enumerate} 
\end{theorem}
\begin{proof} 
$(1)\Rightarrow (2)$. Evident.

$(2)\Rightarrow (3)$. 

$(a)$. By Proposition \ref{P:semireg} $R$ is semiregular. We will prove that $R_P$ is a valuation ring for each maximal ideal $P$. By Proposition \ref{P:locS} $R_P$ is $2$-semiregular for each maximal ideal $P$. So, we may assume that $R$ is a local ring and $P$ its maximal ideal. 

Let $a,b$ be two non-zero elements of $P$. By way of contradiction suppose that $a\notin Rb$ and $b\notin Ra$. Let $A=Ra+Rb$ and $\pi:R^2\rightarrow A$ be the homomorphism defined by $\pi((r,s))=ra-sb$. If $K=\ker(\pi)$, then $K\subseteq PR^2$. By Lemma \ref{L:2F} applied to $M=R/A$ we get that there exists an exact sequence \[0\rightarrow R/A\rightarrow G\rightarrow A\rightarrow 0,\] where $G$ is a free $R$-module of finite rank. By using the fact that $G$ is projective there exists a homomorphism $\phi:G\rightarrow R^2$ such that the following diagram is commutative:
\[\begin{matrix}
{} & {} & {} & {} & {} & {} & 0 & {} & {}\\
{} & {} & {} & {} & {} & {} & \downarrow & {} & {} \\
0 & \rightarrow & R/A & \rightarrow & \ G & \rightarrow & A & \rightarrow & 0 \\
 {} & {} & \downarrow & {} & {\scriptstyle{\phi}}\downarrow & {} & \downarrow & {} & {} \\
0 & \rightarrow & K & \rightarrow & \ R^2 & \xrightarrow{\pi} & A & \rightarrow & 0 \\
{} & {} & {} & {} & {} & {} & \downarrow & {} & {} \\
{} & {} & {} & {} & {} & {} & 0 & {} & {} 
\end{matrix}\]
where  the right vertical map is the identity. Since $K\subseteq PR^2$, $\phi$ is surjective. It follows that the left vertical map is surjective too. Let $L=\ker(\phi)$. Then $L$ is projective and isomorphic to a submodule of $R/A$. It follows that $L=0$, $\phi$ and the left vertical map are isomorphisms. We deduce that $K$ is generated by one element $(s,t)$ which is annihilated by $A$. This implies that $Ra\cap Rb=0$. By \cite[Proposition 4.1]{Mat85} we get $(0:a)+(0:b)=(0:Ra\cap Rb)=R$. But, clearly $(0:a)$ and $(0:b)$ are contained in $P$. We conclude that the ideals $Ra$ and $Rb$ are comparable.

$(b)$. Let $r\in R$ and $A=(0:r)$. By Lemma \ref{L:2F} applied to $M=R/Rr$ we get that there exists an exact sequence \[0\rightarrow R/Rr\rightarrow G\rightarrow Rr\rightarrow 0,\] where $G$ is a projective $R$-module of finite rank. By using the fact that $G$ is projective there exists a homomorphism $\phi:G\rightarrow R$ such that the following diagram is commutative:
\[\begin{matrix}
 {} & {} & {} & {} & {} & {} & 0 & {} & {}\\
{} & {} & {} & {} & {} & {} & \downarrow & {} & {} \\
0 & \rightarrow & R/Rr & \rightarrow & \ G & \rightarrow & Rr & \rightarrow & 0 \\
 {} & {} & \downarrow & {} & {\scriptstyle{\phi}}\downarrow & {} & \downarrow & {} & {} \\
0 & \rightarrow & A & \rightarrow & \ R & \rightarrow & Rr & \rightarrow & 0 \\
{} & {} & {} & {} & {} & {} & \downarrow & {} & {} \\
{} & {} & {} & {} & {} & {} & 0 & {} & {} 
\end{matrix}\]
where  the right vertical map is the identity. Let $C$ be the cokernel of $\phi$. Then $C$ is cyclic and it is isomorphic to the cokernel of the left vertical map. Hence we get the desired property. 

The last assertion is a consequence of Lemma \ref{L:proj}.

$(3)\Rightarrow (1)$. By using Lemma \ref{L:proj} the conditions $(a)$ and $(b)$ imply that $R/Rr$ is $2$-periodic for each $r\in R$. It follows that each direct sum of cyclically presented modules is $2$-periodic. Let $U$ be a finitely presented $R$-module. Since $R$ is arithmetical, $U$ is a direct summand of a module $W=\oplus_{1\leq k\leq n}C_k$, where $C_k$ is a cyclically presented module, for $k=1,\dots,n$. So, there exists an exact sequence $0\rightarrow W\rightarrow F_1\rightarrow F_0\rightarrow W\rightarrow 0$ where $F_0$ and $F_1$ are projective of rank $n$. Let $V$ be an $R$-module satisfying $U\oplus V=W$. There exists two free modules $L$ of rank $p$ and $M$ of rank $q$ and two exact sequences \[0\rightarrow U'\rightarrow L\xrightarrow{\gamma} U\rightarrow 0,\ \ \ \ 0\rightarrow V'\rightarrow M\xrightarrow{\delta} V\rightarrow 0.\] Let $F=L\oplus M$ and $W'=U'\oplus V'$. By Lemma \ref{L:2F} there exists a finitely generated projective $R$-module $G$ and an exact sequence $0\rightarrow W\rightarrow G\xrightarrow{\alpha} F\xrightarrow{\beta} W\rightarrow 0,$ where $\beta$ is induced by $\gamma$ and $\delta$.  By using Diagram \ref{E:diag}, we easily get that $F$ and $G$ have the same rank. Let $G'$ be the inverse image of $U'$ by $\alpha$. We get the following exact sequence: $0\rightarrow U\rightarrow G'/V\rightarrow L\xrightarrow{\gamma} U\rightarrow 0$. We will prove that $G'/V$ is projective and will deduce that $U$ is $2$-periodic. Since $G'/V$ is finitely presented it is enough to show that $(G'/V)_P$ is flat for each maximal ideal $P$. Let $P$ be a maximal ideal of $R$. Then $G_P$ is a free $R_P$-module and $U_P$, $V_P$ and $W_P$ are direct sums of cyclic finitely presented modules. So, there exist a basis $(\epsilon_i)_{1\leq i\leq p+q}$ of $G_P$, a basis $(e_i)_{1\leq i\leq p}$ of $L_P$, a basis $(e_i)_{p+1\leq i\leq p+q}$ of $M_P$, and a family $(r_i)_{1\leq i\leq p+q}$ of elements of $R_P$ such that $\alpha_P(\epsilon_i)=r_ie_i$ for each $i=1,\dots,p+q$. So, $U'_P=\oplus_{1\leq j\leq p}R_Pr_je_j$, $U\cong\oplus_{1\leq j\leq p}R_P/R_Pr_j\cong\oplus_{1\leq j\leq p}R_Ps_j$ where $R_Ps_j=(0:r_j)$ for each $j,\ 1\leq j\leq p$. Then, if $H=\oplus_{1\leq j\leq p}R_P\epsilon_j$, by Lemma \ref{L:summand} $(G'/V)_P\cong H$, and consequently $(G'/V)_P$ is free.
\end{proof}

It is clear using \cite[Theorem 2.1]{EG} that a  B\'ezout ring $R$ is $2$-quasi-Frobenius if and only if it is quasi-Frobenius. The following direct consequence of Theorem \ref{T:2G-sr}  may be seen as the semiregular  counterpart of this result. 

\begin{corollary} Let $R$ be a B\'ezout ring. Then $R$ is $2$-semiregular if and only if $R$ is semiregular.
\end{corollary}

\begin{corollary} \label{C:semilocal} Let $R$ be a semilocal ring. Then the following conditions are equivalent:
\begin{enumerate}
\item $R$ is $2$-semiregular;
\item Every finitely presented cyclic $R$-module is $2$-periodic;
\item $R$ is semiregular and arithmetical.
\end{enumerate}
\end{corollary}
\begin{proof}
$(1)\Rightarrow (2)$ is obvious.

$(2)\Rightarrow (3)$ follows from Theorem \ref{T:2G-sr} ($(2)\Rightarrow (3)(a)$).

$(3)\Rightarrow (1)$. By \cite[(2.3)Corollary]{LLS74} $R$ is an elementary divisor ring. We conclude by Proposition \ref{P:2sr}.
\end{proof}

From Corollary \ref{C:semilocal} and \cite[Theorem 10]{Couch03} we deduce the following:

\begin{corollary} \label{C:Cyc-val} Let $R$ be a local ring which is not a field. The  following assertions are equivalent:
\begin{enumerate}
\item $R$ is $2$-semiregular;
\item every finitely presented cyclic module is $2$-periodic;
\item $R$ is a valuation semiregular ring;
\item $R$ is a valuation ring and there exists $0\ne a\in R$ such that $(0:a)$ is  nonzero and finitely generated;
\item $R$ is a B\'ezout ring and its maximal ideal $P$ is not flat.
\end{enumerate} 
\end{corollary}

\section{Examples and applications}
We end this paper with this section which provides examples and illustrates the limit of some  results given in the previous section.\medskip

Let us start with the direct product of rings.

\begin{proposition}
Let $(R_{\lambda})_{\lambda\in\Lambda}$ be a family of rings and $R=\prod_{\lambda\in\Lambda}R_{\lambda}$. Then $R$ is $2$-semiregular if and only if $R_{\lambda}$ is $2$-semiregular for each $\lambda\in\Lambda$.
\end{proposition}
\begin{proof} For $\lambda\in\Lambda$ let $e_{\lambda}=(\delta_{\lambda,\mu})_{\mu\in\Lambda}$, where $\delta_{\lambda,\mu}$ is the Kronecker symbol. Since $R_{\lambda}\cong R/(1-e_{\lambda})R$, $R_{\lambda}=S^{-1}_{\lambda}R$ where $S_{\lambda}=\{e_{\lambda}\}$. 

So, if $R$ is $2$-semiregular then $R_{\lambda}$ is $2$-semiregular by Proposition \ref{P:locS}.

Conversely, assume that $R_{\lambda}$ is $2$-semiregular for each $\lambda\in\Lambda$. First we show that $R$ is arithmetical\footnote{Any direct product of arithmetical rings is arithmetical too. Possibly this result is already known.}. Let $A,\ B$ and $C$ be ideals of $R$. By \cite[Theorem 1]{Jen66} we have to show that $(A+B)\cap C\subseteq (A\cap C)+(B\cap C)$. If $c=a+b$ with $a=(a_{\lambda})_{\lambda\in\Lambda}\in A,\ b =(b_{\lambda})_{\lambda\in\Lambda}\in B$ and $c =(c_{\lambda})_{\lambda\in\Lambda}\in C$ then, for each $\lambda\in\Lambda,\ c_{\lambda}=a_{\lambda}+b_{\lambda}$. Since $(R_{\lambda}a_{\lambda}+ R_{\lambda}b_{\lambda})\cap R_{\lambda}c_{\lambda}=(R_{\lambda}a_{\lambda}\cap R_{\lambda}c_{\lambda})+(R_{\lambda}b_{\lambda}\cap R_{\lambda}c_{\lambda})$ we have $c_{\lambda}=a'_{\lambda}+ b'_{\lambda}$ with $a'_{\lambda}= s_{\lambda} a_{\lambda}= u_{\lambda}c_{\lambda}$ and $b'_{\lambda}= t_{\lambda} b_{\lambda}= v_{\lambda}c_{\lambda}$ where $s_{\lambda}, u_{\lambda}, t_{\lambda}, v_{\lambda}\in R_{\lambda}$. So, if we put $a'=(a'_{\lambda})_{\lambda\in\Lambda}$ and $b'=(b'_{\lambda})_{\lambda\in\Lambda}$ we get that $c=a'+b'$ with $a'\in Ra\cap Rc$ and $b'\in Rb\cap Rc$. Hence $R$ is arithmetical.

Now, we show that $R$ is coherent. By \cite[Corollary 1.11]{ShWi74} it is enough to check that $(0:r)$ is finitely generated for each $r=(r_{\lambda})_{\lambda\in\Lambda}\in R$. Clearly $(0:r)=\prod_{\lambda\in\Lambda}(0:r_{\lambda})$. By Theorem \ref{T:2G-sr}$(3)(b)$ $(0:r_{\lambda})$ is generated by at most $2$ elements for each $\lambda\in\Lambda$. It follows that $(0:r)$ is generated by at most $2$ elements too. Hence $R$ is coherent.

Let $M$ be a finitely presented $R$-module and $M_{\lambda}=R_{\lambda}\otimes_RM$ for each $\lambda\in\Lambda$. We have $M\cong\prod_{\lambda\in\Lambda}M_{\lambda}$. For each $\lambda\in\Lambda$ there exist two projective $R_{\lambda}$-modules $F_{\lambda},\ G_{\lambda}$ and the following exact sequence:
\[0\rightarrow M_{\lambda}\rightarrow G_{\lambda}\rightarrow F_{\lambda}\rightarrow M_{\lambda}\rightarrow 0.\]
We put $F=\prod_{\lambda\in\Lambda}F_{\lambda}$ and $G=\prod_{\lambda\in\Lambda}G_{\lambda}$. We get the following exact sequence of $R$-modules:
\[0\rightarrow M\rightarrow G\rightarrow F\rightarrow M\rightarrow 0.\]
For each $\lambda\in\Lambda$, $G_{\lambda}$ and $F_{\lambda}$ are also projective over $R$. Since $R$ is coherent $G$ and $F$ are flat. So, $M$ is $2$-F-periodic. By using Lemma \ref{L:2F} it is easy to show that $M$ is $2$-periodic. Whence $R$ is $2$-semiregular. 
\end{proof}

Recall that a ring $D$ is said to be {\bf semihereditary} if each finitely generated ideal is projective. Any semihereditary integral domain is called a {\bf Pr\"ufer} domain.

\begin{proposition}\label{P:2F} Let $D$ be a semihereditary ring, $d\ne 0$ a non-unit of $D$ which is not a zerodivisor and $R=D/dD$. Then each $R$-module is $2$-F-periodic, whence $R$ is $2$-semiregular.
\end{proposition}
\begin{proof} By \cite[Proposition II.15]{Cou03} $R$ is semiregular. Let $M$ be an $R$-module. There exists an exact sequence of $R$-modules \[0\rightarrow L\rightarrow F\rightarrow M\rightarrow 0\] where $F$ is free. It is easy to see that $F=F'/dF'$ and $L=L'/dF'$ where $F'$ is a free $D$-module. Moreover $M\cong F'/L'$. Let $Q$ be the quotient ring of $D$ and $G'$ the submodule of $Q\otimes_DF'$ defined by $G'=d^{-1}L'$. We have $L'=dG'$. Let $G=G'/L'$. Then the multiplication by $d$ induces an epimorphism $\delta:G\rightarrow L$ and $\ker(\delta)=F'/L'\cong M$. We easily see that \[G=\{x\in Q\otimes_DF'/L'\mid dx=0\}.\]
Since $Q\otimes_DF'/L'$ is a FP-injective $D$-module ({\it $Q$ is von Neumann regular, so each $Q$-module is FP-injective over $Q$ and $D$, and each factor of a FP-injective $D$-module is FP-injective too}) it follows that $G$ is FP-injective over $R$. The semiregularity of $R$ implies that $G$ is flat over $R$.
\end{proof}

\begin{corollary}
Let $D$ be a semihereditary ring. Then, for any  elements $d,\ a$ of $D$, $(dD:a)$ and $dD\cap aD$ are generated by at most $3$ elements.
\end{corollary}
\begin{proof}
For each $r\in R$ there exists an idempotent $e_r\in R$ such that $(0:r)=R(1-e_r)$. We may assume that $d\ne 0$. Let $A:=(Rd:a)\cap Re_d$. By applying Theorem \ref{T:2G-sr} and Proposition \ref{P:2F} to the ring $Re_d$ we get that $A/Rd$ is generated by $2$ elements. We have $(Rd:a)=A\oplus R(1-e_d)(1-e_a)$. If $A=Ru+Rv+Rd$ then $(Rd:a)$ is generated by $\{u+(1-e_d)(1-e_a), v, d\}$ and $Rd\cap Ra$ by $\{ua, va, da\}$.
\end{proof}

\begin{remark}
There exist $2$-semiregular rings which are not B\'ezout.
\end{remark}
\begin{proof}
By \cite[Theorem 1]{Swa84}, for each integer $n>0$ there exists a Pr\"ufer domain $D$ with an ideal $I$ generated by at least $(n+1)$ generators. So, if $n>1$ and $d$ a non-zero element of $I$ then $R=D/dD$ is a $2$-semiregular ring which is not a B\'ezout ring.
\end{proof}

\begin{remark}
Some semiregular arithmetical rings are not $2$-semiregular.
\end{remark}
\begin{proof}
Let $n,\ D$ and $I$ be as in the proof of the previous remark. Assume that $I$ is generated by $r_0,r_1,\dots,r_n$. Then $I/r_0D$ is generated by at least $n$ elements. It follows that the minimal number of generators of $I/r_0I$ is $\geq n$. Let $R=D/r_0I$. Then $R$ is semiregular by \cite[Proposition 5.3]{Mat85} and arithmetical. Let $\bar{r_0}=r_0+r_0I$. We have $(0:\bar{r_0})=I/r_0I$. If $n\geq 3$ then $R$ is not $2$-semiregular by Theorem \ref{T:2G-sr}. 
\end{proof}

The following example shows that it is impossible to replace $1$ with $2$ in Proposition \ref{P:fact}

\begin{example}
Let $n\ (\geq 3),\ D,\ I,\ r_0,\dots,r_n$ be as in the proof of the previous remark. We take $R=D/r_0r_1D$ and $A=r_0I/r_0r_1D$. Then $R$ is $2$-semiregular but not $R/A$.
\end{example}

Now we show when a local trivial extension is $2$-semiregular. This result together with the next one allows us to show that the inverse of Corollary \ref{C:2F} does not hold in general.

\begin{proposition}\label{P:triv-2semi}
Let $A$ be a local ring which is not a field, $E$ a non-zero $A$-module and $R=A\propto E$ the trivial ring extension of $A$ by $E$. The following two conditions are equivalent:
\begin{enumerate}
\item $R$ is $2$-semiregular;
\item $A$ and $E$ verify the following properties:
\begin{enumerate} 
\item $A$ is a valuation domain,
\item $E$ is  a divisible torsion uniserial $A$-module such that $(0:e)$ is principal for each $e\in E$.
\end{enumerate}   
\end{enumerate}
\end{proposition}
\begin{proof}
$(1)\Rightarrow (2)$. By  \cite[Theorem 4.16]{AW} $A$ is a valuation domain and $E$ is  a divisible uniserial $D$-module.  Suppose there is   $e\in E$ with $(0:e)=0$. So, in $R$, 
$(0:(0,e))=0\times M$. Then, $M=Re$ since the fact that $R$ is $2$-semiregular implies that $(0:(0:(0,e)))=R(0,e)$, and $M$ is isomorphic to $A$. This shows that for any non unit element $a\neq 0$ of $A$, the element $(a,0)$ of $R$ is a non unit regular element which is absurd. 

Now take any non-zero element $e$ of $E$. $(0:(0,e))= (0:e)\times M$ should be finitely generated. Then, $(0:e)$ is a finitely generated ideal of the valuation domain $A$, as desired.

$(2)\Rightarrow (1)$. Also by  \cite[Theorem 4.16]{AW}, $R$ is a valuation ring. We easily check that $(0:(0,e))$ is a nonzero principal ideal of $R$ for each $0\ne e\in E$. We conclude by Corollary \ref{C:Cyc-val}.
\end{proof}
 
Let $D$ be a valuation domain and $Q$ its quotient field. Recall that a uniserial module $U$ is {\bf standard} if there exist two submodules $J$ and $I$ of $Q$, $I\subset J$, such that $U\cong J/I$. 

\begin{proposition}\label{P:standard}
Let $D$ be a valuation domain, $U$ a divisible torsion uniserial $D$-module. Assume that $(0:u)$ is principal for each $u\in U$. Let $R=D\propto U$ be the trivial ring extension of $D$ by $U$ and let $N=\{(0,u)\mid u\in U\}$. Then the following assertions hold:
\begin{enumerate}
\item $R$ is $2$-semiregular;
\item the following conditions are equivalent:
\begin{enumerate}
\item each $R$-module is $2$-F-periodic;
\item $R/N$ is a $2$-F-periodic $R$-module;
\item $U$ is a standard uniserial $D$-module;
\item $R$ is the homomorphic image of a valuation domain.
\end{enumerate}
\end{enumerate}
\end{proposition}
\begin{proof}
$(1)$. By Proposition \ref{P:triv-2semi}, $R$ is $2$-semiregular.

$(2),\ (c)\Leftrightarrow (d)$ by \cite[Theorem X.6.4]{FuSa01}.

$(2),\ (d)\Rightarrow (a)$. We have $R=D'/I$ where $D'$ is a valuation domain and $I$ a non-zero proper ideal of $D'$. By \cite[Corollary II.14.]{Cou03} necessarily $I=dD'$ where $d$ is a non-zero non-unit element of $D'$. We conclude by Proposition \ref{P:2F}.

$(2),\ (a)\Rightarrow (b)$ is obvious.

$(2),\ (b)\Rightarrow (c)$. Clearly $N$ is the nilradical of $R$, it is a prime ideal, $R_N=Q$ the quotient field of $D$ and $R/N=D$. By applying Lemma \ref{L:2F} to $R/N$ there exists an exact sequence 
\begin{equation}\label{eq:seq}
0\rightarrow R/N\rightarrow G\rightarrow N\rightarrow 0,
\end{equation}
 where $G$ is a flat $R$-module. Let $P$ be the maximal ideal of $R$ and $C$ the image of $R/N$ in $G$. First we show that $G=sG$ for any $s\in P\setminus N$. Since $N\subset sR$ and $N$ is prime we have $sN=N$. So, if $g\in G$ there exists $h\in G$ such that $(g-sh)\in C$. But $(0:s)\subseteq N\subseteq (0:g-sh)$. The FP-injectivity of $G$ implies that $(g-sh)=sg'$ for some $g'\in G$. Now, we prove that $G$ is annihilated by $N$. Let $g\in G$ and $n\in N$. From $N^2=0$ we get that $ng\in C$. If $ng\ne 0$ then $(0:g)=n(0:ng)=nN=0$. This contradicts that $g\in sG$ for each $s\in P\setminus N$. Finally consider $s\in P\setminus N$ and $g\in G$ such that $sg=0$. The flatness of $G$ implies that $g\in (0:s)G\subseteq NG=0$. We conclude that the multiplication by each $s\in R\setminus N$ in $G$ is bijective. Hence $G$ is an $R_N$-module. By tensoring the exact sequence (\ref{eq:seq}) we get that $R_N$ is isomorphic to $G$. Whence $U\cong Q/D$.
\end{proof}

\begin{example} In \cite{Oso91}   Osofsky gives an example of a valuation domain $D$ for which there exists a non-standard uniserial (divisible) $D$-module $U$ with principal annihilators.
So, $R=D\propto U$ is a $2$-semiregular ring, but, by Proposition \ref{P:standard} $D$ is an $R$-module which is not $2$-F-periodic.
\end{example}

\end{document}